\documentclass[reqno]{amsart}

\usepackage{amsmath, amsthm, amsfonts, eucal, amssymb}
\usepackage[english]{babel}
\usepackage[utf8]{inputenc}
\usepackage{graphicx}
\usepackage{layout}
\usepackage[all]{xy}
\usepackage{hyperref}
\usepackage[font=small,labelsep=none]{caption}
\usepackage{subcaption}
\usepackage{pinlabel}

\newtheorem{lem}{Lemma}
\newtheorem{thm}{Theorem}
\newtheorem{cor}{Corollary}[thm]
\newtheorem{prop}{Proposition}

\theoremstyle{definition}

\theoremstyle{remark}

\begin{document}

\title[Knots with meridional essential surfaces of arbitrarily high genus]{Prime Knot complements with meridional essential surfaces of arbitrarily high genus}

\author{Jo\~{a}o Miguel Nogueira}
\address{CMUC, Department of Mathematics, University of Coimbra, Apartado 3008, 3001-454 Coimbra, Portugal\\
nogueira@mat.uc.pt}
\thanks{This work was partially supported by the Centro de Matem\'{a}tica da
Universidade de Coimbra (CMUC), funded by the European Regional
Development Fund through the program COMPETE and by the Portuguese
Government through the FCT - Funda\c{c}\~{a}o para a Ci\^{e}ncia e a Tecnologia
under the project PEst-C/MAT/UI0324/2011. This work was also partially suported by the UTAustin$|$Portugal program CoLab.}



\maketitle

\begin{abstract}
We show the existence of infinitely many prime knots each of which having in their complements meridional essential surfaces with two boundary components and arbitrarily high genus.

\end{abstract}

\section{Introduction}
Since the work of Haken and Waldhausen, it is common to study $3$-manifolds, as knot complements, by their decomposition along surfaces into submanifolds. A very important class of surfaces used in these decompositions are the essential surfaces, which has motivated research on the properties and existence of closed essential surfaces or meridional essential surfaces in knot complements in $S^3$. A particularly interesting phenomena is the existence of knots with the property that their complements have closed essential surfaces of arbitrarily high genus. The first examples of knots with this property were given by Lyon \cite{Lyon}, where he proves the existence of fibered knots complements with closed essential surfaces. Later Oertel \cite{Oertel} and recently Li \cite{Li} also give examples of knots having closed essential surfaces of arbitrarily high genus in their complements. Oertel uses the planar surfaces from the tangles defining the Montesino knots to construct and characterize the essential surfaces. Lyon and Li use connected sum of knots on their constructions and afterwards sattelite knots to obtain primeness of the desired examples.\\
In this paper we consider meridional surfaces instead, and prove that there is also no general bound for the genus of meridional essential surfaces in the complements of (prime) knots. In fact, we construct prime knots each of which with meridional essential surfaces in their complements having only two boundary components and arbitrarily high genus. Then, in particular, we prove that some prime knots have the property that they can be decomposed by surfaces of all positive genus as composite knots are decomposed by spheres. The results of this paper are summarized in the following theorem and its corollary.

\begin{thm}\label{prime}
There are infinitely many prime knots each of which having the property that its complement has a meridional essential surface of genus $g$ and two boundary components for all positive $g$.
\end{thm}

\begin{cor}\label{non prime}
There are infinitely many knots each of which having the property that its complement has a meridional essential surface of genus $g$ and two boundary components for all $g\geq 0$.
\end{cor}

\noindent From \cite{CGLS}, at least one of the swallow-follow surfaces obtained from the meridional essential surfaces in Theorem \ref{prime} is also essential and of higher genus. Hence, the knots from the theorem are also examples of knots having closed essential surfaces of arbitrarily high genus in their complements.\\   
Together with the literature already cited, this paper joins several other contributions to unserstand better knots with respect to the existence of closed or meridional essential surfaces in their complements: In \cite{Menasco} Menasco studies essential surfaces in alternating links complements; Finkelstein-Moriah \cite{Finkelstein-Moriah} and Lustig-Moriah \cite{Lustig-Moriah} prove the existence of meridional essential and closed essential surfaces, respectively, for a large class of links characterized by a certain $2n$-plat projection; There is also the work of Finkelstein \cite{Finkelstein} and Lozano-Przytycki \cite{Lozano} describing closed incompressible surfaces in closed $3$-braids; More recently, after Gordon-Reid \cite{Gordon-Reid} proved that tunnel number one knots have no meridional planar essential surface in their complements, Eudave-Muñoz \cite{Munoz1}, \cite{Munoz2} proved that some of these knots actually have meridional or closed essential surfaces in their complements.\\
The proof of Theorem \ref{prime} follows a similar philosophy as in Lyon's paper \cite{Lyon}, where he uses the connected sum of two knots and essential surfaces in their exteriors. We could follow the same construction if our aim was only to construct a knot exterior with merdidional essential surfaces of arbitrarily high genus. However, as we want the surfaces to have two boundary components we cannot use the sattelite construction to obtain primeness of the knot. As we also want the knots to be prime we cannot use a connected sum as the base for the construction. So, instead of using composite knots we consider a decomposition of prime knots along certain essential tori separating the knot into two arcs.
The main techniques for the proof are classical in $3$-manifold topology, as innermost curve arguments and  branched surface theory. The reference used for standard definitions and notation in knot theory is Rolfsen's book \cite{Rolfsen}. Throughout this paper all submanifolds are assumed to be in general position and we work in the piecewise linear category.

\section{Construction of the knots}

In our construction we use \textit{$2$-string essential free tangles}, that we define as follows:
A \textit{$n$-string tangle} is a pair $(B, \sigma)$ where $B$ is a $3$-ball and $\sigma$ is a collection of $n$ properly embedded disjoint arcs in $B$. We say that $(B, \sigma)$ is \textit{essential} if for every disk $D$ properly embedded in $B-\sigma$ then $\partial D$ bounds a disk in $\partial B-\partial \sigma$. The tangle is said to be \textit{free} if the fundamental group of $B-\sigma$ is free, or, equivalently, if the closure of $B-N(\sigma)$ is a handlebody.\\
Let $H$ be a solid torus and $\gamma$ an embedded graph in $H$, as in Figure \ref{Gamma}.
\begin{figure}[htbp]
\labellist
\small \hair 2pt
\pinlabel (a) at 0 0
\scriptsize
\pinlabel $\gamma$ at 5 17
\pinlabel $a_1$ at 23 70
\pinlabel $a_2$ at 23 14

\small
\pinlabel (b) at 118 0
\scriptsize
\pinlabel $\gamma$ at 125 39.5
\pinlabel $H$ at 150 9
\pinlabel $D_H$ at 186.3 39.5
\pinlabel $B_H$ at 215 30
\endlabellist
\centering
\includegraphics{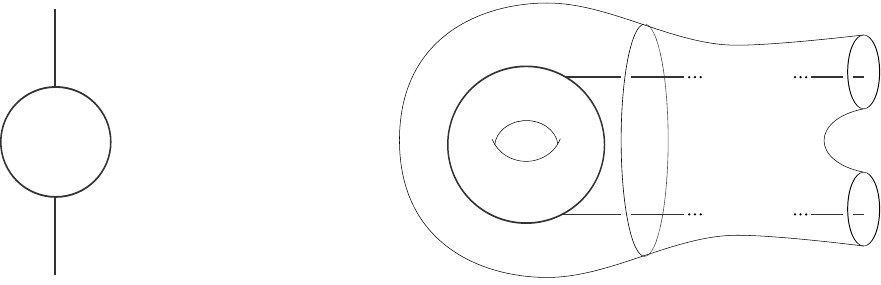}
\caption{: The graph $\gamma$, in (a), and its embedding into the solid torus $H$, in (b).}
\label{Gamma}
\end{figure}
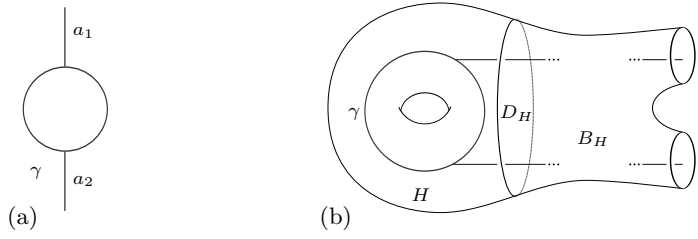
The graph $\gamma$ is topologically a circle connected to two segments, $a_1$ and $a_2$, at a boundary point of each. The other two boundary points of $a_1\cup a_2$ are in $\partial H$. There is a separating disk $D_H$ in $H$ intersecting $\gamma$ transversely at a point of each segment $a_1$ and $a_2$, and decomposing $H$ into a solid torus and a $3$-ball $B_H$ where $(B_H, B_H\cap \gamma)$ is a $2$-string essential free tangle\footnote{See the Appendix, section \ref{appendix}, for an example of a $2$-string essential free tangle with both strings knotted.} with $B_H\cap \gamma$ two knotted arcs in $B_H$. (See Figure \ref{Gamma}(b).)\\
Denote by $T$ a regular neighborhood of $\gamma$ in $H$ and suppose there is a properly embedded arc $s$ in $T$, as in Figure \ref{Gammai}(a), with the boundary of $s$ in $T\cap \partial H$.
\begin{figure}[htbp]
\labellist
\small \hair 2pt
\pinlabel (a) at 3 0
\scriptsize
\pinlabel $Q$ at 41 71
\pinlabel $T$ at 40 5
\pinlabel $D_T$ at 71 40
\pinlabel $R_T$ at 56 25
\pinlabel $B_T$ at 100 27
\pinlabel $s$ at 9 23

\small
\pinlabel (b) at 182 0
\scriptsize
\pinlabel $L_i$ at 221 59
\pinlabel $T$ at 220 5
\pinlabel $D_T$ at 250 40
\pinlabel $R_T$ at 235 25
\pinlabel $B_T$ at 279 27
\pinlabel $s_i$ at 188 23
\endlabellist
\centering
\includegraphics{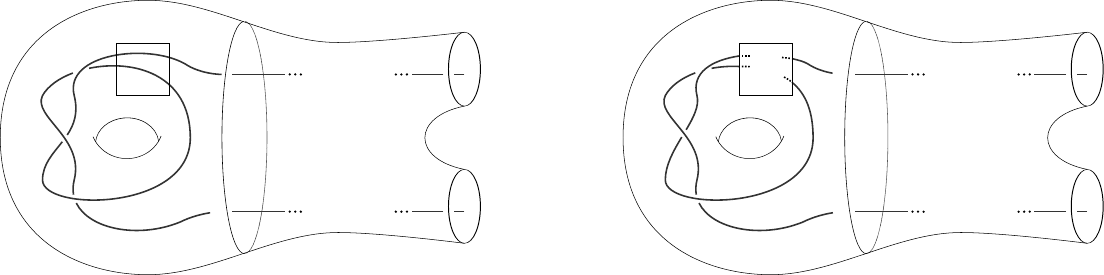}
\caption{: The solid torus $T$ with the string $s$, in (a), and the solid torus $T_i$ with the string $s_i$, in (b).}
\label{Gammai}
\end{figure}
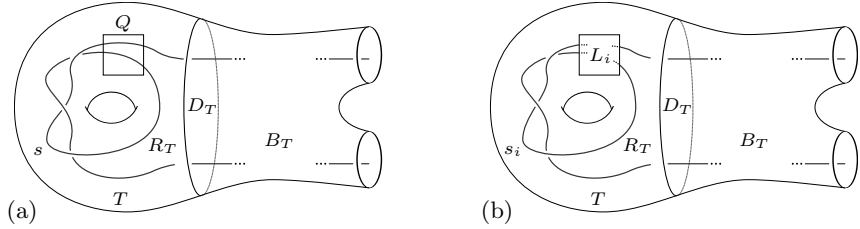
Assume there is a separating disk $D_T$ in $T$ intersecting $s$ at two points and decomposing $T$ into a $3$-ball $B_T$ and a solid torus $R_T$. The boundary of $s$ is in $\partial B_T$ and $(B_T, B_T\cap s)$ is a $2$-string essential free tangle with the two arcs $B_T\cap s$ in $B_T$ being knotted. The string $R_T\cap s$ in the solid torus $R_T$ is such that when capped off by an arc in $D_T$ we get the $(2, -3)$-torus knot boundary parallel in $R_T$.  \\
We say that an arc properly embedded in a solid torus is \textit{essential} if it is not \textit{boundary parallel}, that is the arc does not co-bound an embedded disk in the solid torus with a segment in the boundary of the solid torus, and if the boundary of the solid torus is incompressible in the complement of the arc. In Lemma \ref{essential} we prove that $s$ is essential in $T$.\\
Consider a ball $Q$ in $T-B_T$ intersecting $s$ at two parallel trivial arcs, as in Figure \ref{Gammai}(a), and an infinite collection of knots $L_i$, $i\in \mathbb{N}$. We replace the two parallel trivial arcs by two parallel arcs with the pattern\footnote{By a properly embedded arc in a ball $B$ having the \textit{pattern} of a knot $K$ we mean that when we cap off the arc with a string in $\partial B$ we get the knot $K$.} of a knot $L_i$, as in Figure \ref{Gammai}(b). After this tangle replacement, we denote by $s_i$ the string obtained from $s$, by $T_i$ the solid torus $T$ containing $s_i$, by $\gamma_i$ the graph $\gamma$ whose regular neighborhood is $T_i$, and by $H_i$ the solid torus $H$ containing $T_i$.
Let $E_H(T)$ be the exterior of $T$ in $H$, that is the closure of $H-T$, and $E_T(s)$ be the exterior of $N(s)$ in $T$, that is the closure of $T-N(s)$. The following lemmas are relevant for the next section, and they are also valid if we replace $s$ by $s_i$, $T$ by $T_i$ and $H$ by $H_i$ in their statements.

\begin{lem}\label{essential}
\begin{itemize}
\item[]
\item[(a)] The surfaces $\partial H$ and $\partial T$ are incompressible in $E_H(T)$.
\item[(b)] The arc $s$ is essential in $T$.
\end{itemize}
\end{lem}
\begin{proof}\text{}\\
(a) First we prove that $\partial H$ is incompressible in $E_H(T)$. As $T$ is a regular neighborhood of $\gamma$ this is equivalent to prove that $\partial H$ is incompressible in $H-\gamma$. The graph $\gamma$ in $H$ is defined by a circle $c$ and two segments $a_1$ and $a_2$, each with an end in the circle and the other end in $\partial H$, and $H$ is a regular neighborhood of $c$. Hence, the boundary of a properly embedded disk $D$ in $H$ disjoint from $c$ bounds a disk $O$ in $\partial H$. Furthermore, as $D$ is disjoint from $\gamma$ and each segment $a_1$ and $a_2$ intersects $\partial H$ at a single point, the disk $O$ is disjoint from $\gamma$. Then, the boundary of every embedded disk in $E_H(T)$ with boundary in $\partial H$ bounds a disk in $\partial H-\partial H\cap \partial T$, which means $\partial H$ is incompressible in $E_H(T)$.\\
We prove similarly that $\partial T$ is incompressible in $E_H(T)$. Let $D$ be a properly embedded disk in $E_H(T)$ with boundary in $\partial T$. We have $T=N(c)\cup N(a_1)\cup N(a_2)$. As $a_1$ and $a_2$ have each an end in $\partial H$ and in $c$, we can isotope the boundary of $D$ to $N(c)$. As $H$ is a regular neighborhood of $c$ we have that $\partial D$ bounds a disk $O$ in $\partial N(c)$. As $a_1$ and $a_2$ have each only one end in $c$, we have that $O$ is a disk in $\partial T$. Hence, $\partial T$ is incompressible in $E_H(T)$.\\

\noindent
(b) To prove that $s$ is essential in $T$ we have to prove that $\partial T$ is incompressible in $E_T(s)$ and that $s$ is not boundary parallel. We start by showing that $\partial T$ is incompressible in $E_T(s)$. As the tangle $(B_T, B_T\cap s)$ is essential, the boundary of a properly embedded disk in $B_T- (s\cap B_T)$ bounds a disk in $\partial B_T-(s\cap \partial B_T)$. Also, the string $R_T\cap s$ in the solid torus $R_T$ when capped off by an arc in $D_T$ is the $(2, -3)$-torus knot boundary parallel in $R_T$. Hence, every disk in $R_T-R_T\cap s$ with boundary in $\partial T$ has boundary bounding a disk in $\partial T$. Also, if a disk in $R_T-R_T\cap s$ has boundary in $D_T$ then its boundary bounds a disk in $D_T$ disjoint from $s\cap D_T$. Suppose $D$ is a disk properly embedded in $E_T(s)$ with boundary in $\partial T$. If $D$ is disjoint from $D_T$ then $\partial D$ bounds a disk in $\partial T-s\cap\partial T$. So, if $D$ is a compressing disk for $\partial T$ in $E_T(s)$ it intersects $D_T$. Hence, we assume that $D$ intersects $D_T$ transversely in a collection of arcs and simple closed curves, with $|D\cap D_T|$ minimal. If $D$ intersects $D_T$ in simple closed curves then consider an innermost one in $D$ and the respective innermost disk $O$. From the previous observations, we have that $\partial O$ bounds a disk in $D_T$. Therefore, by an isotopy of $D$ along the ball bounded by $D\cup D_T$ we can reduce $|D\cap D_T|$, which contradicts its minimality. Hence, $D\cap D_T$ is a collection of arcs. Consider an outermost arc $\alpha$ between the arcs $D\cap D_T$ in $D$ and the respective outermost disk, that we also denote by $O$. If $O$ is in $B_T$ then $\partial O$ bounds a disk $O'$ in $\partial B_T$ intersecting $D_T-s$ at a disk. Suppose now that $O$ is in $R_T$. If $O$ is essential in $R_T$ then $O$ intersects at least twice the $(2, -3)$-torus knot obtained from $R_T\cap s$ by capping off the ends of this string in $D_T$. However, $\partial O$ intersects at most once this knot, whether $\alpha$ separates the components of $D_T\cap s$ in $D_T$ or not. This implies that $O$ intersects $R_T\cap s$, which is contradiction with $O$ being disjoint from $s$. Therefore, $O$ is inessential in $R_T$ and $\partial O$ bounds a disk $O'$ in $\partial R_T$ intersecting $D_T-s$ at a disk. In both cases, $O$ in $B_T$ or in $R_T$, $\partial O$ bounds a disk $O'$ intersecting $D_T-s$ at a disk. If we isotope $D$ along the ball bounded by $O\cup O'$ we reduce $|D\cap D_T|$, contradicting its minimality. Hence, $\partial T$ is incompressible in $E_T(s)$.\\
Now we prove that $s$ is not boundary parallel in $T$. Suppose that $D$ is now a disk embedded in $T$ co-bounded by $s$ and an arc $b$ in $\partial T$. Following a similar argument as before we can prove that $D$ does not intersect $D_T$ at simple closed curves and arcs with both ends in $b$. Hence, $D\cap D_T$ is a collection of two arcs, each with an end in $s$ and the other end in $b$. However, the disk components these arcs separate from $D$ imply that the strings of the tangle $(B_T, B_T\cap s)$ are trivial, which contradicts this tangle being essential. Hence, $s$ is not boundary parallel in $T$ and, together with $\partial T$ being incompressible in the exterior of $s$ in $T$, we have that $s$ is essential in $T$.
\end{proof}

\begin{lem}\label{disk in H}
There is no properly embedded disk in $E_H(T)$
\begin{itemize}
\item[(a)] intersecting one of the disks of $T\cap \partial H$ at a single point; or
\item[(b)] with boundary the union of an arc in $\partial T$ and an arc in $\partial H$, and not bounding a disk in $\partial E_H(T)$.
\end{itemize}
\end{lem}
\begin{proof}
Let $D$ be a properly embedded disk in $E_H(T)$. Following an argument as in Lemma \ref{essential} we can assume that $|D\cap D_H|$ is minimal and that $D\cap D_H$ is a collection of essential arcs in $D_H-D_H\cap T$ with ends in $D_H\cap T$ and $\partial D_H$. Consider also $B_H\cap T$, which is a collection of two cylinders $C_1$ and $C_2$, and assume that $|D\cap (B_H\cap T)|$ is minimal. If some arcs of $\partial D\cap C_i$ have ends in the same boundary component of the annlus $\partial C_i-C_i\cap\partial B_H$, $i=1,2$,  then by using an innermost curve argument we can reduce $|D\cap (B_H\cap T)|$ and contradict its minimality. Therefore, $D\cap C_i$ is a collection of essential arcs in the annulus $\partial C_i-C_i\cap\partial B_H$, $i=1, 2$.\\

\noindent (a) Assume that $D$ intersects $T\cap \partial H$ exactly once at $C_1\cap\partial H$. As $D\cap D_H$ is a collection of arcs, the components of $D\cap B_H$ are a collection of disks. Consequently, one component of $D\cap B_H$ is a disk in $B_H-B_H\cap T$ intersecting $C_1\cap \partial H$ once. This means that $C_1\cap \partial H$ is primitive with respect to the complement of $B_H\cap T$ in $B_H$. Hence, as the complement of $C_1\cup C_2$ is a handlebody (because $(B_H, B_H\cap \gamma)$ is a free tangle), the complement of $C_2$ in $B_H$ is a solid torus. Then the core of $C_2$ is unknotted, which is a contradiction to $B_H\cap \gamma$ being a collection of two knotted arcs in $B_H$.\\

\noindent
(b) Suppose the disk $D$ is as in the statement with $\partial D=a\cup b$, where $a$ is an arc in $\partial T$ and $b$ an arc in $\partial H$. 
The intersection of $D$ with $T\cap \partial H$ is the boundary of $a$ (and $b$), and notice that $C_1\cup C_2$ intersects $\partial H$ at $T\cap\partial H$. From the statement $(a)$ of this lemma the boundary of $a$ (and $b$) is in one disk component of $T\cap\partial H$, that without loss of generality we assume to be $C_1\cap \partial H$. As $D\cap C_i$ is a collection of essential arcs in $\partial C_i-C_i\cap \partial B_H$ and $\partial D$ intersects $\partial T\cap \partial H$ in two points, the arc $a$ intersects $D_H$ at two points. Furthermore, by an innermost arc argument and the minimality of $|D\cap D_H|$, there are no arcs of $D\cap D_H$ with both end points in $b\cap D_H$.

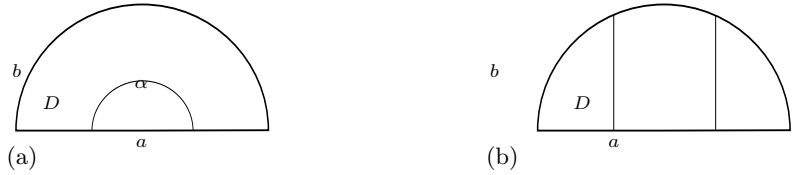
\begin{figure}[htbp]
\labellist
\small \hair 2pt
\pinlabel (a) at 17 -3
\scriptsize
\pinlabel $a$ at 62 3
\pinlabel $\alpha$ at 62 25
\pinlabel $D$ at 28 18
\pinlabel $b$ at 15 30

\small
\pinlabel (b) at 198 -3
\scriptsize
\pinlabel $a$ at 240 3
\pinlabel $D$ at 228 18
\pinlabel $b$ at 195 30

\endlabellist
\centering
\includegraphics{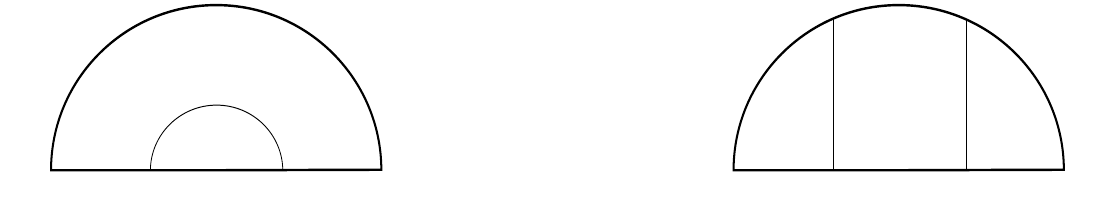}
\caption{: The arcs $D\cap D_H$ in $D$ when a single arc intersects $a$, in (a), and when two arcs intersect $a$, in (b).}
\label{DBoundTH}
\end{figure}

\noindent
The points of $a\cap D_H$ cobound one or two arcs of $D\cap D_H$. Suppose $a\cap D_H$ is the boundary of a single arc, $\alpha$, component of $D\cap D_H$, as in Figure \ref{DBoundTH}(a). Hence, the arc $\alpha$ separates a disk from $D$ in $B_H$ that cuts a ball from $B_H-C_1$ contaning $C_2$, otherwise $\partial D$ bounds a disk in $\partial E_H(T)$. As there are no local knots in the free tangle $(B_H, B_H\cap s)$, this case implies that either the core of $C_2$ corresponds to a trivial string or a string parallel to the core of $C_1$, which contradicts this tangle being a $2$-string essential free tangle (see Lemma 2.1 in \cite{Nogueira}). If each point of $a\cap D_H$ cobounds an arc of $D\cap D_H$ with the other end in $b$, as in Figure \ref{DBoundTH}(b), then the core $C_1$ corresponds to a string of $(B_H, B_H\cap \gamma)$ that is trivial, which also contradicts this tangle being essential.
\end{proof}

\noindent To construct the knots to prove the main theorem of this paper we identify the solid tori $H$ and $H_i$ along their boundaries, by identifying a meridian of one boundary to a longitude of the other and defining a Heegaard decomposition $H\cup H_i$ of $S^3$, such that $\partial s$ is identified with $\partial s_i$. From construction, $K_i=s\cup s_i$ is a knot in $S^3$, for $i\in \mathbb{N}$, and in the next proposition we prove these knots are prime.

\begin{prop}\label{Ki}
The knots $K_i$, $i\in \mathbb{N}$, are an infinite collection of distinct prime knots.
\end{prop}

\noindent To prove that the knots $K_i$ are prime we use the following technical result. Let $K$ and $L$ be non-trivial knots. Take a ball $B$ intersecting $K$ in two parallel trivial arcs with the tangle $(B^c, B^c\cap K)$ being locally unknotted. We replace the arcs of $B\cap K$ in $B$ by two parallel arcs with the pattern of $L$, where each new arc has the same boundary component as one of the replaced arcs. We denote the resulting knot from this construction by $K_L$.

\begin{lem}\label{KL}
The knot $K_L$ is prime.
\end{lem}
\begin{proof}
If the knot $K_L$ is trivial then it bounds a disk $D$ in $S^3$. Then, $\partial D$ intersects $\partial B$ at four points. Suppose that $|D\cap \partial B|$ is minimal. By an innermost curve argument, as used before, we can show that $D\cap \partial B$ is a collection of two arcs. The strings of $B\cap K_L$ are knotted and each can't co-bound an outermost disk of $D-D\cap \partial B$ with an arc in $\partial B$. Hence, the arcs of $D\cap \partial B$ have an end on each string of $B\cap K_L$ and co-bound together with the strings a disk in $B$. Each arc of $D\cap \partial B$ also co-bounds a disk with a string of $K_L\cap B^c$. Therefore, if we replace the tangle $(B, B\cap K_L)$ with the tangle $(B, B\cap K)$ we obtain a disk in $S^3$ bounded by $K$, which is a contradiction because $K$ is knotted. Hence, the knot $K_L$ is also non-trivial.\\
Now we prove that $K_L$ is prime. Suppose there is a decomposing sphere $S$ for $K_L$. As $(B, B\cap K_L)$ is defined by two parallel strings in $B$, using the disk co-bounded by the two strings $B\cap K_L$ in $B$ we can show that $S$ can be assumed disjoint from $B$. However, this means that $S$ is in $B^c$, which contradicts $(B^c, B^c\cap K_L)$ being locally unknotted.
\end{proof}

\noindent As for the construction of the knots $K_i$, we construct a knot $K$ by identifying two copies of $H$, say $H$ and $H'$, by identifying a meridian of one boundary to a longitude of the other and defining a Heegaard decomposition $H\cup H'$ of $S^3$, such that the two copies of $s$, say $s$ and $s'$ resp., are also identified along their boundaries. As $s$ is essential in $H$ we have that $\partial H$ defines a meridional incompressible surface in the exterior of $K$, which means that $K$ is not trivial. We also denote the copy of the solid torus $T$ of $H$ in $H'$ by $T'$. \\
We will use this knot $K$, the knots $L_i$ and the construction of Lemma \ref{KL} to define the knots $K_i$, but first we need the following lemma. Let $Q$ be the ball as in Figure \ref{Gammai} and $Q^c$ its complement.

\begin{lem}\label{locally}
The tangle $(Q^c, Q^c\cap K)$ is locally unknotted.
\end{lem}
\begin{proof}
Suppose $(Q^c, Q^c\cap K)$ is locally knotted. Then there is a sphere $S$ bounding a ball $P$ intersecting $Q^c\cap K$ at a single knotted arc. We have that $s$ and $s'$ have no local knots in $T\cup T'$. Then $S$ intersects $T$ or $T'$.\\
Consider the intersection of $S$ with $\partial T$ and $\partial T'$, and  suppose it has a minimal number of components. From the construction of the knot $K$ the cores of the solid tori $T$ and $T'$ define a two component link with each component being unknotted.\\
As the tangle $(B_T, B_T\cap s)$ is free and essential we can assume that $S$ is disjoint from $B_T$ (and similarly, that $S$ is disjoint from $B_{T'}$).\\
The intersections of $S$ with the boudaries of $T$ and $T'$ is a collection of simple closed curves. As $S$ is disjoint from $B_T$ and $B_{T'}$ the curves of intersection are either in $\partial T-B_T$ or in $\partial T'-B_{T'}$. Consider $E$ a disk component of $S$ separated by $\partial T\cup \partial T'$ from $S$. Suppose $E$ is not in $T\cup T'$ and its boundary is in $\partial T$ (or similarly $\partial T'$). From the minimality of $|S\cap \partial(T\cup T')|$ and as $S^3$ does not have a $S^2\times S^1$ or a Lens space summand, we have that $\partial E$ is a longitude of $\partial T$. Therefore, the core of $T$ bounds a disk disjoint from $T'$, which is a contradiction to the cores of $T$ and $T'$ being linked from construction. Hence, $E$ is in $T$ or $T'$. If $E$ is in $T$ (or similarly in $T'$) and is disjoint from $s$ then as $s$ is essential in $T$ we have that $\partial E$ bounds a disk in $\partial T-s$. In this case we can reduce the number of components of $S$ intersection with $\partial T\cup \partial T'$, which is a contradiction to its minimality. Then, we can assume that all disks $E$ intersect $s$ or $s'$. If some disk $E$ intersects either $s$ or $s'$ at two points then some other disk component of $S$ separated by $\partial T\cup \partial T'$ is disjoint from $s$ and $s'$, which is a contradiction to all disks $E$ intersecting $s$ or $s'$. Then, there is an essential disk $E$ in $T$ (or similarly, in $T'$) that intersects $s$ at a single point. As before, let $R_T$ be the solid torus separated by $D_T$ from $T$. From the construction of $s$ in $T$, if we cap off $R_T\cap s$ with an arc in $D_T$ we get a torus knot. Then any essential disk in $R_T$ intersects the knot in more than one point. As $E$ is disjoint from $B_T$ it is a non-separating disk in $R_T$ intersecting the torus knot at a single point, which is a contradiction. Hence, $(Q^c, Q^c\cap K)$ is locally unknotted.
\end{proof}

\begin{proof}[Proof of Proposition \ref{Ki}]
The knots $K_i$ are the knots $K_{L_i}$ obtained from the knots $K$ and $L_i$ with a construction as in Lemma \ref{KL}. From Lemmas \ref{KL} and \ref{locally} the knots $K_i$ are prime.\\
Each knot $K_i$ is also sattelite with companion knot $L_i$ and pattern knot $K$. Then, from the unicity of JSJ-decomposition of compact $3$-manifolds and as the knots $L_i$ are distinct we have that the knots $K_i$, $i\in \mathbb{N}$, are an infinite collection of distinct prime knots.
\end{proof}

\section{Knots with meridional essential surfaces for all genus}

In this section we prove Theorem \ref{prime}, and its corollary, by showing the knots $K_i$, $i\in \mathbb{N}$, have meridional essential surfaces of all positive genus and two boundary components. We start by constructing these surfaces, denoted by $F_1, \ldots F_g, \ldots$ where $F_g$ has genus $g$, in the complement of an arbitrary knot $K_i$, and afterwards we prove they are essential in $E(K_i)$. In this construction we denote the boundaries of $s$ and $s_i$ by $\partial_1 s$ ($=\partial_1 s_i$) and $\partial_2 s$ ($=\partial_2 s_i$). Denote by $X$ (resp., $Y$) the punctured torus $\partial T$ (resp., $\partial T_i$) obtained by cutting the interior of the discs $\partial T\cap \partial H$ (resp., $\partial T_i\cap \partial H$) . We also denote by $\partial_i X$ (resp., $\partial_i Y$) the boundary component of $X$ (resp., $Y$) related to $\partial_i s$, $i=1, 2$.\\
The surface $F_1$ is defined as $X$ together with the annuli cut by $\partial X$ from $\partial H\cap E(K_i)$, that we denote by $O_i$, $i=1, 2$, with respect to $\partial_i X$. The surface $F_2$ is obtained from $X$ and $Y$ by gluing two copies of $O_1$ to $\partial_1 X$ and $\partial_1 Y$, pushing them slightly into $H$ and $H_i$ respectively, and  identifying the boundary components $\partial_2 X$ and $\partial_2 Y$. In Figure \ref{F1} we have a schematic representation of $F_1$ and $F_2$.

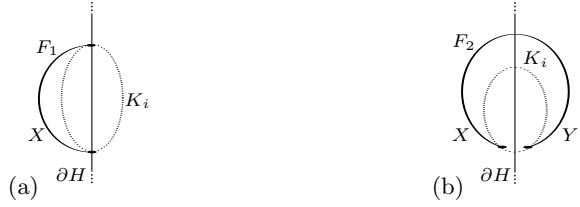
\begin{figure}[htbp]
\labellist
\small \hair 2pt
\pinlabel (a) at 55 0
\scriptsize
\pinlabel $F_1$ at 64 53
\pinlabel $X$ at 60 20
\pinlabel $K_i$ at 98 33
\pinlabel $\partial H$ at 73 5

\small
\pinlabel (b) at 215 0
\scriptsize
\pinlabel $F_2$ at 221 55
\pinlabel $X$ at 220 20
\pinlabel $Y$ at 261 20
\pinlabel $K_i$ at 248 48.5
\pinlabel $\partial H$ at 233 5

\endlabellist
\centering
\includegraphics{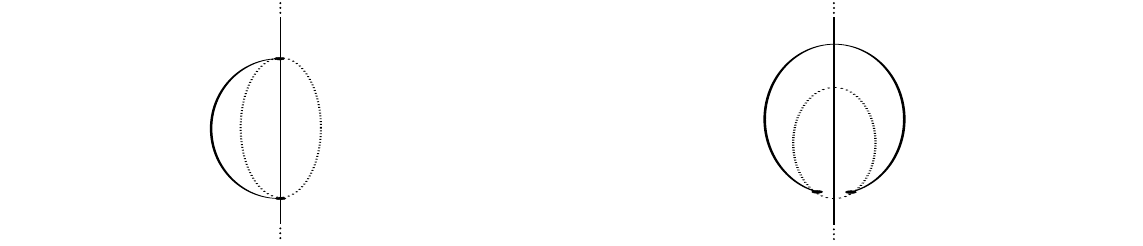}
\caption{: A schematic diagram of surface $F_1$, in (a), and surface $F_2$, in (b).}
\label{F1}
\end{figure}

\noindent To construct the surfaces $F_g$, for $g\geq 3$, we follow a general procedure as explained next. In $H_i$ consider a copy of $Y$ and an annulus $A$, around $s_i$, defined by $\partial N(s_i)-(\partial N(s_i)\cap \partial H_i)$. We denote by $Z$ the surface obtained by identifying $Y$ and $A$ along the boundaries $\partial_1 Y$ and $\partial_1 A$. Let $n=g-1$ and $A_1, \ldots, A_{n-2}$ be disjoint copies of $A$ disjoint from $Z$. Consider also $n$ disjoint copies of $X$ in $H$, denoted by $X_1, \ldots, X_n$. Denote $\partial_1 X_j$ (resp., $\partial_2 X_j$) the boundary component of $X_j$ around $\partial_1 s$ (resp., $\partial_2 s$). Similarly, we label the boundary components of $A_j$ by $\partial_1 A_j$ and $\partial_2 A_j$. To construct $F_g$ we start by attaching $\partial_2 X_n$ and $\partial_2 X_{n-1}$ to the two boundary components of $Z$ respecting the order from $\partial_2 s$. If $g\geq 4$ we also attach $\partial_2 X_{n-2}, \ldots, \partial_2 X_1$ to $\partial_2 A_{n-2}, \ldots, \partial_2 A_1$, respectively, and $\partial_1 X_{n}, \ldots, \partial_1 X_3$ to $\partial_1 A_{n-2}, \ldots, \partial_1 A_1$, respectively. The surface $F_g$ has two boundary components ($\partial_1 X_1$ and $\partial_1 X_2$) and Euler characteristic $-2g$, which means the genus of $F_g$ is $g$. In Figure \ref{F3} we have a schematic representation of $F_3$ and $F_4$, and in Figure \ref{Fn} a representation of the general construction of $F_g$.\\

\begin{figure}[htbp]
\labellist
\small \hair 0pt
\pinlabel (a) at 45 0
\scriptsize
\pinlabel $F_3$ at 60 80
\pinlabel $Z$ at 97 82
\pinlabel $X_2$ at 62 21
\pinlabel $X_1$ at 66 30
\pinlabel $K_i$ at 91 60
\pinlabel $\partial H$ at 73 5

\small
\pinlabel (b) at 195 0
\scriptsize
\pinlabel $F_4$ at 221 90
\pinlabel $Z$ at 258 90
\pinlabel $X_3$ at 218 12
\pinlabel $X_2$ at 222 21
\pinlabel $X_1$ at 225 30
\pinlabel $A_1$ at 266 51
\pinlabel $K_i$ at 249 62
\pinlabel $\partial H$ at 233 5

\endlabellist
\centering
\includegraphics{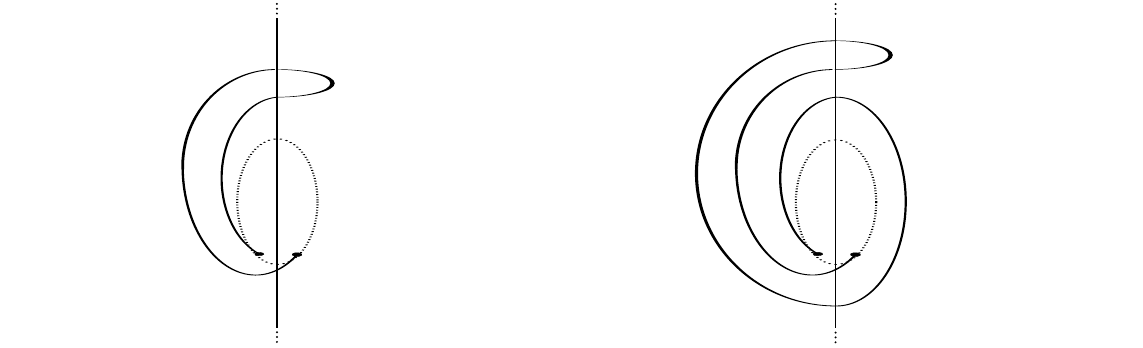}
\caption{: A schematic diagram of surface $F_3$, in (a), and surface $F_4$, in (b). }
\label{F3}
\end{figure}
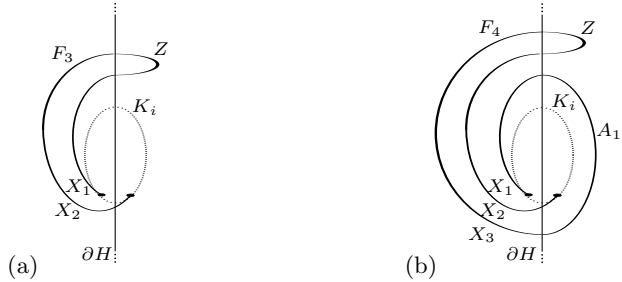

\begin{figure}[htbp]

\labellist
\hair 2pt
\scriptsize
\pinlabel $F_g$ at 128 143
\pinlabel $Z$ at 180 145

\pinlabel $X_{g-1}$ at 133 16
\pinlabel $X_{g-2}$ at 139 24
\pinlabel $X_{g-3}$ at 141 32
\pinlabel $X_3$ at 137 46
\pinlabel $X_2$ at 138 56
\pinlabel $X_1$ at 141 72
\pinlabel $\cdots$ at 118 75
\pinlabel $A_1$ at 192 72
\pinlabel $\cdots$ at 205 75
\pinlabel $A_{g-3}$ at 210 25
\pinlabel $K_i$ at 175 89
\pinlabel $\partial H$ at 154 5

\endlabellist

\centering
\includegraphics{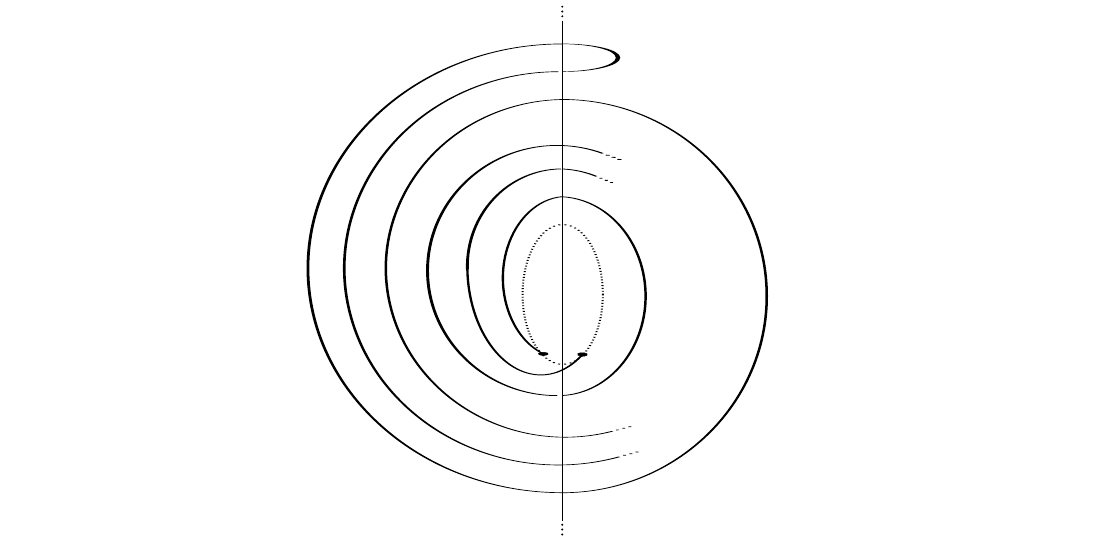}
\caption{: A schematic general representation of the surface $F_g$, for $g\geq 3$.}
\label{Fn}
\end{figure}

\begin{lem}\label{g=1,2}
The surfaces $F_1$ and $F_2$ are essential in the exterior of the knot $K_i$, $i\in \mathbb{N}$.
\end{lem}
\begin{proof}
For $g=1, 2$ we assume that $F_g$ is not essential in $E(K_i)$ and prove this leads to a contradiction using innermost curve arguments. We denote generically by $D$ a compressing or a boundary compressing disk for $F_g$ in $E(K_i)$. In case $D$ is a boundary compressing disk then $\partial D=a\cup b$ where $a$ is an arc in $\partial E(K_i)-F_g$, with one end in each component of $\partial F_g$, and $b$ is an arc in $F_g$. We also assume $|D\cap \partial H|$ to be minimal. Consequently, using an innermost curve argument, as in Lemma \ref{essential}, we have that $D$ does not intersect $\partial H$ in simple closed curves.\\

\noindent Suppose $g=1$. By a small isotopy of a neighborhood of $\partial F_1$ into $H$ if necessary, we can assume that $F_1$ is in $H$. If $D$ is a compressing disk for $F_1$ in $E(K_i)$ then $D\subset\; H$, as $D$ cannot intersect $\partial H$ in simple closed curves and $\partial D$ is disjoint from $\partial H$. This is a contradiction to Lemma \ref{essential}, which says $\partial T$ is incompressible in $E_T(s)$ and in $E_H(T)$. Assume now $D$ is a boundary compressing disk of $F_1$ in $E(K_i)$. If $D$ is in $T$ then we have a contradiction to Lemma \ref{essential}(b) for $s$ being essential in $T$. If $D$ is not in $T$ then, by using an innermost curve argument, we can assume that $a$ intersects $\partial H$ at two points and that $D\cap \partial H$ is an arc separating from $D$ a disk $O$ in $H$ with boundary an arc in $\partial H$ and an arc that we can assume in $\partial T$ having ends in $\partial H\cap \partial T$. Hence, $O$ contradicts Lemma \ref{disk in H}(b). Therefore, we have that $F_1$ is essential in $E(K_i)$.\\

\noindent Suppose $g=2$. By a small isotopy of a neighborhood of $\partial F_2$ we can assume that the component of $\partial F_2\cap X$ is in $H$ and that $\partial F_2\cap Y$ is in $H_i$. Suppose $D$ is a compressing disk of $F_2$ in $E(K_i)$. If $D$ is disjoint from $\partial H$ then $D$ is a compressing disk for $X$ or $Y$ in $E(K_i)$, which is a contradiction to Lemma \ref{essential}(a). Then, assume $D$ intersects $\partial H$ at a minimal collection of arcs. Consider an outermost arc $\alpha$ of $D\cap \partial H$ in $D$ and let $O$ be the respective outermost disk, with $O\cap F_2=\beta$ an arc in $X$ or in $Y$. Without loss of generality, suppose $\beta$ is in $X$. If $\alpha$ or $\beta$ does not co-bound a disk in $\partial H$ or $X$, respectively, with $\partial_2 X$ we have a contradiction to Lemma \ref{disk in H}(b). Otherwise, $\partial O$ bounds a disk $O'$ in $\partial H\cup \partial T$ and using the ball bounded by $O\cup O'$ we can isotope $D$ reducing $|D\cap \partial H|$ which is a contradiction to its minimality.\\
Suppose now that $D$ is a boundary compressing disk for $F_2$ in $E(K_i)$. As the two components of $\partial F_2$ are in opposite sides of $\partial H$ by an innermost curve argument we can prove that $a$ intersects $\partial H$ at a single point. Hence, $D\cap \partial H$ is an arc with one end in $a$ and one end in $b$ and, possibly, arcs with both ends in $b$, as in Figure \ref{g2}.\\
\begin{figure}[htbp]

\labellist
\hair 2pt
\scriptsize

\pinlabel $a$ at 150 1
\pinlabel $D$ at 130 18
\pinlabel $b$ at 125 45

\endlabellist

\centering
\includegraphics{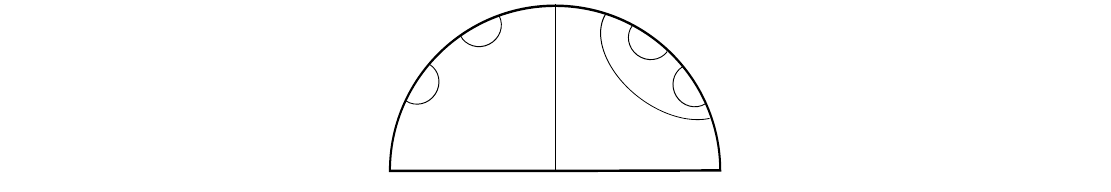}
\caption{: Arcs of $D\cap \partial H$ when $D$ is a boundary compressing disk of $F_2$.}
\label{g2}
\end{figure}

\noindent If $D$ is in $T\cup T_i$ then the arcs of $D\cap \partial H$ with both ends in $b$ are in the annulus $O_2$. Hence, each arc of $D\cap \partial H$ with ends in $b$ co-bounds a disk in $O_2$ with $\partial_2 X$. Consider an outermost of such arcs in $O_2$, and the respective outermost disk $O$. By cutting and pasting $D$ along $O$ we contradict Lemma \ref{essential}(b) or we can reduce $|D\cap \partial H|$ contradicting its minimality. Therefore, in this case, $D\cap \partial H$ is an arc with an end in $a$. This arc cuts $D$ into two disks, one in $T$ and the other in $T_i$, contradicting Lemma \ref{essential}(b). If $D$ is in $E_H(T)\cup E_{H_i}(T_i)$ we consider an outermost arc $\alpha$ between the arcs of $D\cap \partial H$ in $D$ and the respective outermost disk, also denoted by $O$. If the arc $\beta$, that is $\partial O\cap F_2$, co-bounds a disk in $F_2$ with $F_2\cap \partial H$, using an argument as before, we can reduce $|D\cap \partial H|$ contradicting its minimality. Otherwise, the disk $O$ is in contradiction to Lemma \ref{disk in H}(b). Hence, $D\cap \partial H$ is only an arc with an end in $a$, and the disk separated by this arc in $D$ is also in contradiction to Lemma \ref{disk in H}(b). Consequently, $F_2$ is essential in $E(K_i)$.
\end{proof}

To prove the surfaces $F_g$, $g\geq 3$, are essential in the complement of the knots $K_i$, we use branched surface theory.  First, we start by revising the definitions and result relevant to this paper from Oertel's work in \cite{Oertel2}, and also Floyd and Oertel's work in \cite{Floyd-Oertel}.\\
A \textit{branched surface} $B$ with generic branched locus is a compact space locally modeled on Figure \ref{branchedmodel}(a). Hence, a union of finitely many compact smooth surfaces in a $3$-manifold $M$, glued together to form a compact subspace of $M$ respecting the local model, is a branched surface. We denote by $N=N(B)$ a fibered regular neighborhood of $B$ (embedded) in $M$, locally modelled on Figure \ref{branchedmodel}(b). The boundary of $N$ is the union of three compact surfaces $\partial_h N$, $\partial_v N$ and $\partial M\cap \partial N$, where a fiber of $N$ meets $\partial_h N$ transversely at its endpoints and either is disjoint from $\partial_v N$ or meets $\partial_v N$ in a closed interval in its interior. We say that a surface $S$ is \textit{carried} by $B$ if it can be isotoped into $N$ so that it is transverse to the fibers. Furthermore, $S$ is carried by $B$ with \textit{positive weights} if $S$ intersects every fiber of $N$. If we associate a weight $w_i\geq 0$ to each component on the complement of the branch locus in $B$ we say that we have an $\textit{invariant measure}$ provided that the weights satisfy \textit{branch equations} as in Figure \ref{branchedmodel}(c). Given an invariant measure on $B$ we can define a surface carried by $B$, with respect to the number of intersections between the fibers and the surface. We also note that if all weights are positive then the surface carried can be isotoped to be transverse to all fibers of $N$, and hence is carried with positive weights by $B$.\\

\begin{figure}[htbp]

\labellist
\small \hair 0pt
\pinlabel (a) at 3 -5

\pinlabel (b) at 173 -5

\pinlabel (c) at 335 -5

\pinlabel  $\partial \text{ }N$ at 207 23
\pinlabel \tiny $h$ at 205 21

\pinlabel  $\partial \text{ }N$ at 155 40
\pinlabel \tiny $v$ at 153 37

\pinlabel $w_3=w_2+w_1$ at 383 50
\pinlabel $w_1$ at 407 14
\pinlabel $w_2$ at 391 30
\pinlabel $w_3$ at 368 30

\endlabellist
\centering
\includegraphics[width=0.9\textwidth]{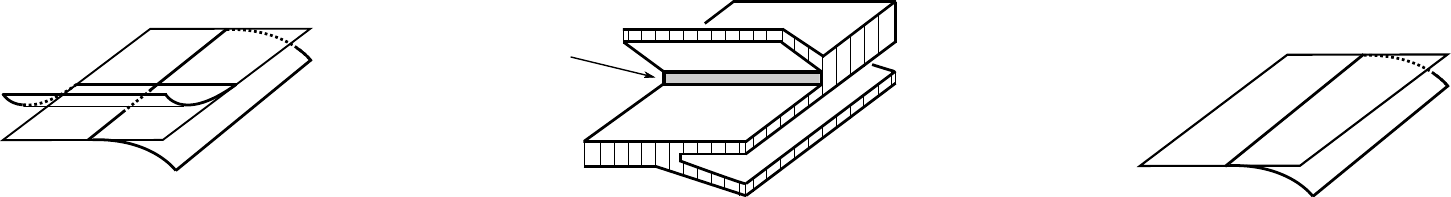}
\caption{: Local model for a branched surface, in (a), and its regular neighborhood, in (b).}
\label{branchedmodel}
\end{figure}
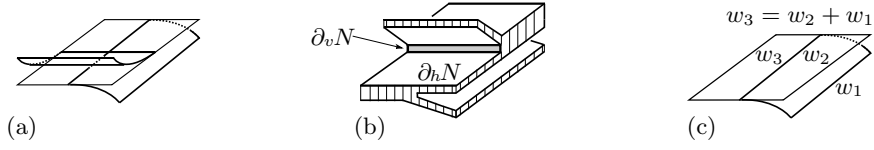

A \textit{disc of contact} is a disc $D$ embedded in $N$ transverse to fibers and with $\partial D\subset \partial_v N$. A \textit{half-disc of contact} is a disc $D$ embedded in $N$ transverse to fibers with $\partial D$ being the union of an arc in $\partial M\cap \partial N$ and an arc in $\partial_v N$. A \textit{monogon} in the closure of $M-N$ is a disc $D$ with $D\cap N=\partial D$ which intersects $\partial_v N$ in a single fiber. (See Figure \ref{monogon}.)\\

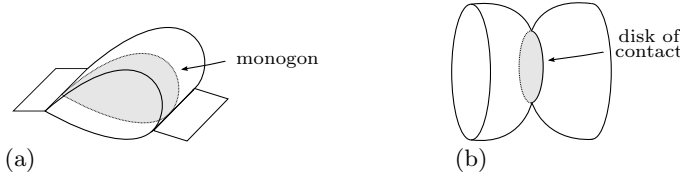
\begin{figure}[htbp]

\labellist
\small \hair 0pt
\pinlabel (a) at 3 -7

\scriptsize
\pinlabel monogon at 100 30

\small
\pinlabel (b) at 173 -7

\scriptsize
\pinlabel  \text{disk of} at 240 40

\pinlabel \text{contact} at 240 34

\endlabellist
\centering
\includegraphics{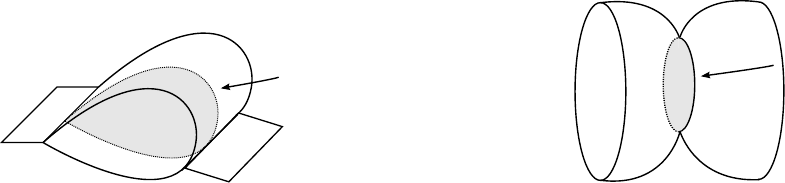}
\caption{: Illustration of a monogon and a disk of contact on a branched surface.}
\label{monogon}
\end{figure}

A branched surface embedded $B$ in $M$ is said \textit{incompressible} if it satisfies the following three properties:
\begin{itemize}
\item[(i)] B has no disk of contact or half-disks of contact;
\item[(ii)] $\partial_h N$ is incompressible and  boundary incompressible in the closure of $M-N$, where a boundary compressing disk is assumed to have boundary defined by an arc in $\partial M$ and an arc in $\partial_h N$;
\item[(iii)] There are no monogons in the closure of $M-N$. 
\end{itemize}

The following theorem proved by Floyd and Oertel in \cite{Floyd-Oertel} let us infer if a surface carried by a branched surface is essential.

\begin{thm}[Floyd and Oertel, \cite{Floyd-Oertel}]\label{Floyd-Oertel}
A surface carried with positive weights by an incompressible branched surface is essential.
\end{thm}

We now prove that the remaining surfaces $F_g$ are essential.

\begin{lem}\label{g=3}
The surfaces $F_g$, $g\geq 3$, are essential in the exterior of the knot $K_i$, $i\in \mathbb{N}$.
\end{lem}
\begin{proof}
To prove the statement of this theorem, we construct a branched surface that carries $F_g$, $g\geq 3$, and show that it is incompressible in the exterior of $K_i$, $i\in \mathbb{N}$.\\ 

Let us consider the puntured torus $X$ in $H$, the annulus $A$ in $H_i$, the annulus $O_1$ and the punctured torus $Y$ in $H_i$. Note that the boundaries of $\partial_i X$, $\partial_i Y$ and $\partial_i A$, for $i=1, 2$, are the same. Consider the union $X\cup A\cup Y$ and isotope $\partial_1 Y$, in this union, into the interior of $A\cap H_i$. Now we add $O_1$ to the previous union and denote the resulting space by $B$. We smooth the space $B$ on the intersections of the surfaces $X$, $A$, $Y$ and $O_1$ as follows: the annulus $O_1$ in its intersection with $X\cup A$ is smoothed in the direction of $X$; the punctured torus $Y$ on its boundary $\partial_1 Y$ is smoothed in the direction of $\partial_2 A$, and on its boundary $\partial_2 Y$ is smoothed in the direction of $X$. We keep denoting the resulting topological space by $B$. From the construction, the space $B$ is a branched surface with sections denoted naturally by $X$, $O_1$, $A'$, $A$ and $Y$, as illustrated in Figure 3. We denote a regular neighborhood of $B$ by $N(B)$.\\

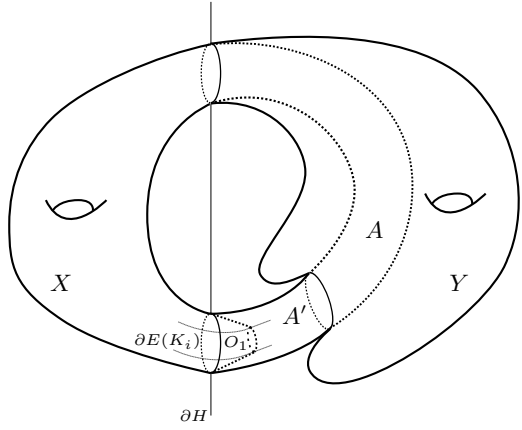
\begin{figure}[htbp]
\labellist
\hair 0pt
\small

\pinlabel $X$ at 20 50
\tiny
\pinlabel  $\partial E(K_i)$ at 60 28
\small
\pinlabel $Y$ at 170 50

\pinlabel $A$ at 138 70
\pinlabel $A'$ at 108 38
\tiny
\pinlabel $O_1$ at 86 27.5

\pinlabel $\partial H$ at 70 0

\endlabellist

\centering
\includegraphics{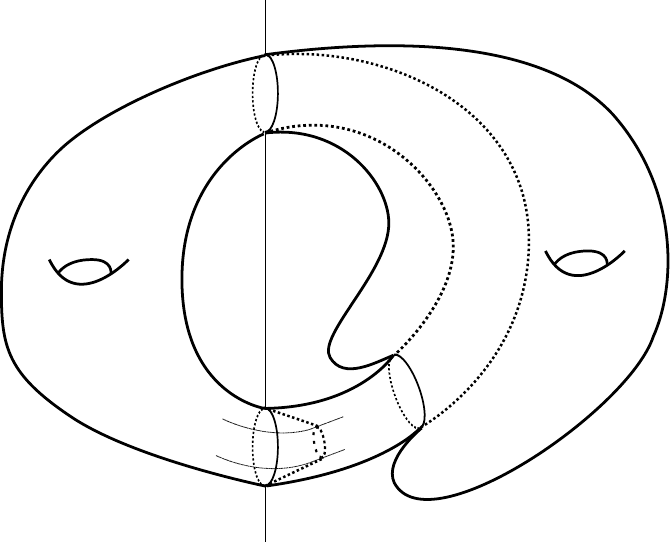}
\caption{: The branched surface $B$ in $E(K_i)$.}
\label{B}
\end{figure}

Given the invariant measure on $B$ defined by $w_X=g-1$, $w_{O_1}=2$, $w_{A'}=g-3$, $w_A=g-2$ and $w_Y=1$, the weights for the sections $X$, $O_1$, $A'$, $A$ and $Y$ respectively, we have that, for each $g\geq 3$, the surface $F_g$ is carried with positive weights by $B$.\\
To prove that $F_g$, $g\geq 3$, is essential in the complement of $E(K_i)$ we show that $B$ is an incompressible branched surface in $E(K_i)$ and use Theorem \ref{Floyd-Oertel}.\\   
The space $N(B)$ decomposes $E(K_i)$ into three components: a component cut from $E(K_i)$ by $X\cup O_1\cup A'\cup A$ that we denote $E_X$; a component cut from $E(K_i)$ by $Y\cup A$ that we denote by $E_Y$; a component cut from $E(K_i)$ by $Y\cup A'\cup X$ that we denote by $E_B$. Note that $\partial E_X\cap \partial N(B)$ is ambient isotopic to $\partial T$ in $E(K_i)$. Hence, from Lemma \ref{essential}(b), we have that $\partial_h N(B)$ is incompressible and boundary incompressible in $E_X$, and also that there are no monogons in $E_X$. Similarly, $\partial E_Y\cap \partial N(B)$ is ambient isotopic to $\partial T_i$ in $E(K_i)$. As $\partial_2 Y$ corresponds to the only component of $\partial_v N(B)$ in $E_Y$, a monogon in $E_Y$ corresponds to the arc $s_i$ being trivial in $T_i$. Therefore, from Lemma \ref{essential}, there are no monogons in $E_Y$, and $\partial_h N(B)$ is incompressible and boundary incompressible in $E_Y$. At last, we consider the component $E_B$, which corresponds to gluing $E_H(T)$ and $E_{H_i}(T_i)$ along their boundaries as before. Suppose there is a compressing disk $D$ for $\partial E_B$ in $E_B$. Note also that $\partial E_B$ is a ambient homotopic to $X\cup Y$ identified along their boundaries. As in the proof of Lemma \ref{g=1,2}, we assume $|D\cap \partial H|$ to be minimal and that the intersection $D\cap \partial H$ contains no simple closed curves. If $D$ is disjoint from $\partial H$ then $\partial D$ is a compressing disk for $X$ in $H$ or for $Y$ in $H_i$, which contradicts Lemma \ref{essential}(a). Then, $D$ intersects $\partial H$ in a collection of arcs, as in Figure \ref{gb3}.

\begin{figure}[htbp]
\labellist
\small \hair 2pt

\scriptsize
\pinlabel $\alpha$ at 54 48
\pinlabel $O$ at 44 52
\pinlabel $D$ at 69 30
\pinlabel $\beta$ at 34 56

\endlabellist

\centering
\includegraphics{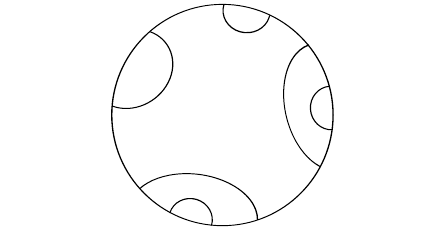}
\caption{: The disk $D$ together with the arcs $D\cap \partial H$.}
\label{gb3}
\end{figure}
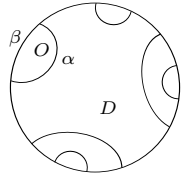

Consider an outermost arc in $D$ between the arcs of $D\cap \partial H$ and denote it by $\alpha$. Let $O$ be the outermost disk cut from $D$ by $\alpha$, and let $\partial O=\alpha\cup \beta$ where $\beta$ is an arc in $\partial E_B$. Consequently, the disk $O$ is in $H$ or in $H_i$. If $\beta$ co-bounds a disk in $\partial N(B)$ with an arc in $\partial H$ we can reduce $|D\cap \partial H|$ and contradict its minimality. Then, $O$ is a compressing disk of $\partial E_H(T)$ in $E_H(T)$. As $\partial O$ has boundary defined by the union of an arc in $\partial T$ and an arc in $\partial H$, we have a contradiction to Lemma \ref{disk in H}(b). 
\end{proof}

\begin{proof}[Proof of Theorem \ref{prime}.]
From Lemma \ref{g=1,2} and \ref{g=3} we have that the surfaces $F_g$, $g\in \mathbb{N}$, are essential in the complements of the knots $K_i$, $i\in \mathbb{N}$. Together with Lemma \ref{Ki}, we obtain the statement of the theorem.
\end{proof}

\noindent The proof Corollary \ref{non prime}  now follows naturally.

\begin{proof}[Proof of Corollary \ref{non prime}]
In Theorem \ref{prime} we proved that the knots $K_i$, $i\in \mathbb{N}$, are an infinite collection of prime knots with meridional essential surfaces in their complements for each positive genus and two boundary components. Hence, considering the knots $K_i$ connected sum with some other knot, we have infinitely many knots with meridional essential surfaces of genus $g$ and two boundary components for all $g\geq 0$.
\end{proof}

\section{Appendix}\label{appendix}

In this appendix we give an example of a $2$-string essential free tangle with both strings knotted.\\
For a string $s$ in a ball $B$ we can consider the knot obtained by capping off $s$ along $\partial B$, that is by gluing to $s$ an arc in $\partial B$ along the respective boundaries. We denote this knot by $K(s)$. The string $s$ is said to be \textit{knotted} if the knot $K(s)$ is not trivial.\\
Let $s_1$ be an arc in a ball $B$ such that $K(s_1)$ is a trefoil, and consider also an unknotting tunnel $t$ for $K(s_1)$, as in Figure \ref{FTangles1}.

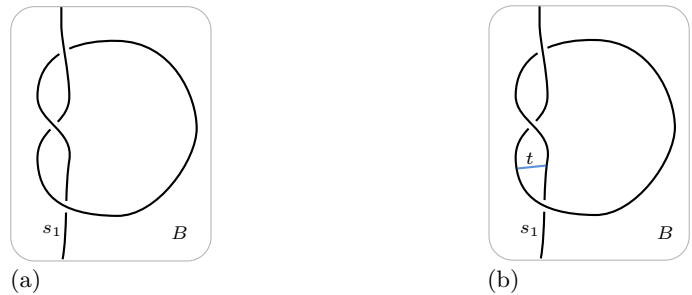
\begin{figure}[htbp]
\labellist
\small \hair 2pt
\pinlabel (a) at 37 0
\scriptsize
\pinlabel $s_1$ at 47 19
\pinlabel $B$ at 95 18

\small
\pinlabel (b) at 217 0
\scriptsize
\pinlabel $s_1$ at 227 19
\pinlabel $t$ at 227 46
\pinlabel $B$ at 278 18

\endlabellist

\centering
\includegraphics{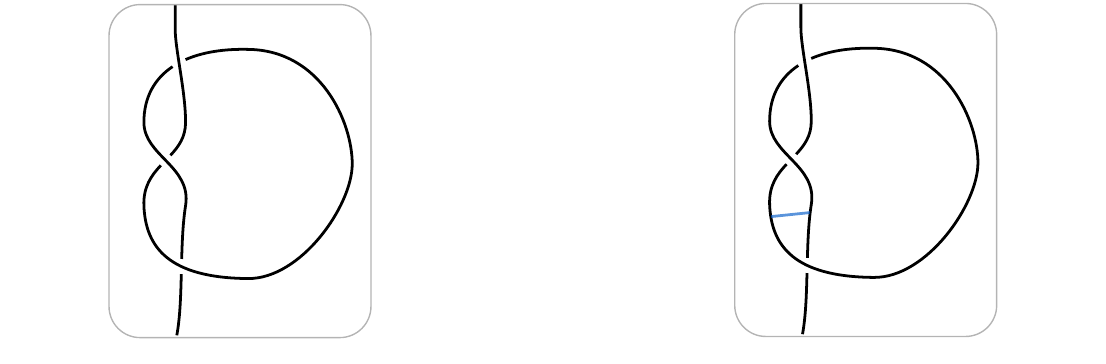}
\caption{: The string $s_1$ when capped off along $\partial B$ is a trefoil knot with an unknotting tunnel $t$.}
\label{FTangles1}
\end{figure}

\noindent If we slide $\partial t$ along $s_1$ into $\partial B$, as illustrated in Figure \ref{FTangles2}(a), we get a new string that we denote by $s_2$, as in Figure \ref{FTangles2}(b).

\begin{figure}[htbp]
\labellist
\small \hair 2pt
\pinlabel (a) at 37 0
\scriptsize
\pinlabel $s_1$ at 47 19
\pinlabel $t$ at 48 47
\pinlabel $B$ at 95 18

\small
\pinlabel (b) at 217 0
\scriptsize
\pinlabel $s_1$ at 227 19
\pinlabel $s_2$ at 227 46
\pinlabel $B$ at 278 18

\endlabellist

\centering
\includegraphics{FTangles2}
\caption{: Construction of a $2$-string essential free tangle, with both strings knotted, from $s_1$ and the unknotting tunnel $t$.}
\label{FTangles2}
\end{figure}

\noindent The knot $K(s_2)$ is the $(3,-4)$-torus knot, and hence knotted. The tangle $(B, s_1\cup s_2)$ is free by construction. In fact, as $t$ is an unknotting tunnel of $K(s_1)$, the complement of $N(s_1)\cup N(t)$ in $B$ is a handlebody. Henceforth, by an ambient isotopy, the complement of $N(s_1)\cup N(s_2)$ is also a handlebody. As the tangle $(B, s_1\cup s_2)$ is free and both strings are knotted then it is necessarily essential. Otherwise, the complement of $N(s_1)\cup N(s_2)$ in $B$ is not a handlebody as it is obtained by gluing two non-trivial knot complements along a disk in their boundaries, which is a contradiction to the tangle $(B, s_1\cup s_2)$ being free.

\section*{Acknowledgement}

The author would like to thank John Luecke, for introductory conversations on branched surface theory, and the referee, for an insightful and educative review.

\end{document}